\documentclass[a4paper,11pt,twoside,english]{article}
\usepackage{babel}
\usepackage{latexsym,amsfonts}
\usepackage{amsthm}
\usepackage{amsmath}
\usepackage{mathtools}
\usepackage{paralist}

\usepackage{lastpage}
\usepackage{fancyhdr}
\fancypagestyle{plain}{

	\fancyhead[L,C]{}
	\fancyhead[R]{}
	\fancyfoot[R,L]{}
	\fancyfoot[C]{\thepage \ of \pageref{LastPage}}
}

\title{Locally octahedral and locally almost square K\"othe-Bochner spaces}
\author{Jan-David Hardtke}
\date{}

\providecommand{\N}{\mathbb{N}}
\providecommand{\R}{\mathbb{R}}
\providecommand{\A}{\mathcal{A}}
\providecommand{\B}{\mathcal{B}}
\providecommand{\Po}{\mathcal{P}}
\providecommand{\ssq}{\subseteq}
\providecommand{\sm}{\setminus}
\providecommand{\eps}{\varepsilon}

\providecommand{\keywords}[1]{
{\let\thefootnote=\relax
\footnote{{\em Keywords}: #1}}
\addtocounter{footnote}{-1}
}

\providecommand{\AMS}[1]{
{\let\thefootnote=\relax
\footnote{{\em AMS Subject Classification} (2020): #1}}
\addtocounter{footnote}{-1}
}

\providecommand{\address}{
{\sc \noindent Department of Mathematics \\
Universit\"at Leipzig \\
Augustusplatz 10, 04109 Leipzig \\
Germany \\}
}

\DeclarePairedDelimiter{\norm}{\lVert}{\rVert}
\DeclarePairedDelimiter{\set}{\lbrace}{\rbrace}

\DeclarePairedDelimiter{\abs}{\lvert}{\rvert}

\theoremstyle{definition}
\newtheorem{definition}{Definition}[section]
\newtheorem*{definition*}{Definition}

\newtheorem*{example*}{Example}

\newtheorem*{remark*}{Remark}
\theoremstyle{plain}
\newtheorem{lemma}[definition]{Lemma}
\newtheorem*{lemma*}{Lemma}

\newtheorem*{proposition*}{Proposition}
\newtheorem{theorem}[definition]{Theorem}
\newtheorem*{theorem*}{Theorem}

\newtheorem*{corolary*}{Corollary}

\newenvironment{Abstract}{\centering\begin{minipage}{0.8\textwidth} \noindent \small {\sc Abstract.}}{\end{minipage}\par}

\usepackage{color}
\definecolor{darkgreen}{rgb}{0,0.5,0}

\numberwithin{equation}{section}
\newtagform{colored}[\color{blue}]{\color{blue}(}{\color{blue})}
\usetagform{colored}

\usepackage[colorlinks,linkcolor=blue,citecolor=red,urlcolor=darkgreen]{hyperref}
\providecommand{\email}{{\it E-mail address:} \href{mailto:hardtke@math.uni-leipzig.de}{\tt hardtke@math.uni-leipzig.de}}

\usepackage{amsrefs}

\begin{document}
	
\maketitle

\begin{Abstract}
It has been proved in \cite{hardtke} that a K\"othe-Bochner space $E(X)$ is locally octahedral/locally almost square
if $X$ has the respective property and the simple functions are dense in $E(X)$.\par
Here we show that the result still holds true without the density assumption. The proof makes use of the Kuratowski-Ryll-Nardzewski Theorem on measurable selections.
\end{Abstract}
\AMS{46B20 46E30 46E40 28B20}

\section{Introduction}\label{sec:introduction}
Let $X$ be a real Banach space. We denote by $X^*$ its topological dual, by $B_X$ its closed unit ball and by $S_X$ its unit sphere.\par
$X$ is locally octahedral (LOH) if the following holds: for every $x\in S_X$ and every $\eps>0$ there exists $y\in S_X$ such that $\norm{x\pm y}\geq 2-\eps$. This notion was introduced in \cite{haller3} in connection with the so called diameter-two-properties.\par
$X$ is locally almost square (LASQ) if the following holds: for every $x\in S_X$ and every $\eps>0$ there exists $y\in S_X$ such that $\norm{x\pm y}\leq 1+\eps$. This notion was introduced in \cite{abrahamsen2}.\par
For more information on these and related properties the reader may consult \cites{abrahamsen,abrahamsen2,abrahamsen3,acosta,becerra-guerrero,becerra-guerrero2,becerra-guerrero4,becerra-guerrero5,becerra-guerrero3,godefroy,haller,haller2,haller3,haller4,haller5,hardtke2,hardtke,kubiak,langemets1,langemets2,lopez-perez}, \cite{deville}*{Theorem 2.5, p.106} and references therein.\par
Now consider a complete, $\sigma$-finite measure space $(S,\A,\mu)$. For a set $A\subseteq S$ we denote by $\chi_A$ the characteristic function of $A$. Let $(E,\norm{\cdot}_E)$ be a Banach space of real-valued measurable functions on $S$ (modulo equality $\mu$-almost everywhere) such that the following holds:
\begin{enumerate}[(i)]
\item $\chi_A\in E$ for every set $A\in \A$ with $\mu(A)<\infty$.
\item If $f\in E$ and $A\in \A$ with $\mu(A)<\infty$, then $f$ is $\mu$-integrable over $A$.
\item If $g$ is measurable and $f\in E$ such that $\abs*{g(t)}\leq\abs*{f(t)}$ $\mu$-a.\,e. then $g\in E$
and $\norm{g}_E\leq\norm{f}_E$.
\end{enumerate}
Then $(E,\norm{\cdot}_E)$ is called a K\"othe function space over $(S,\A,\mu)$.

\newpage

Standard examples are the spaces $L^p(\mu)$ for $1\leq p\leq\infty$. More generally, Orlicz spaces with the Luxemburg norm (see \cite{rao}) are examples of K\"othe function spaces.\par
A function $f:S \rightarrow X$ is called simple if there are finitely many pairwise disjoint sets $A_1,\dots ,A_n\in \A$ such that $\mu(A_i)<\infty$ for all $i=1,\dots,n$, $f$ is constant on each $A_i$ and $f(t)=0$
for every $t\in S\sm \bigcup_{i=1}^nA_i$. \par 
$f$ is said to be Bochner-measurable if there exists a sequence $(f_n)_{n\in \N}$ of simple functions such that $\lim_{n\to \infty}\norm{f_n(t)-f(t)}=0$ $\mu$-a.\,e.\par 
According to the well-known Pettis Measurability Theorem, the following assertions are equivalent:
\begin{enumerate}[\upshape(i)]
\item $f$ is Bochner-measurable.
\item $x^*\circ f$ is measurable for every $x^*\in X^*$ and there is a separable, closed subspace $Y$ of $X$ such that $f(s)\in Y$ for $\mu$-a.\,e. $s\in S$.
\item $f$ is measurable (with respect to $\A$ and the Borel-$\sigma$-algebra of $X$) and there is a separable, closed subspace $Y$ of $X$ such that $f(s)\in Y$ for $\mu$-a.\,e. $s\in S$.
\end{enumerate}	

We denote by $E(X)$ the space of all Bochner-measurable functions
$f:S\rightarrow X$ (modulo equality a.\,e.) such that $\norm{f(\cdot)}\in E$.\par 
We define $\norm{f}_{E(X)}=\norm*{\norm{f(\cdot)}}_E$ for $f\in E(X)$. Then $E(X)$ becomes a Banach space, the so called K\"othe-Bochner space induced by $E$ and $X$.\par 
For $E=L^p(\mu)$ we obtain the usual Lebesgue-Bochner spaces $L^p(\mu,X)$ for $1\leq p\leq\infty$. For more information on K\"othe-Bochner spaces the reader is referred to the book \cite{lin}.\par
In \cite{hardtke}*{Theorems 4.1 and 4.5} the author proved the following results:
\begin{enumerate}[(a)]
\item If $X$ is LOH/LASQ and the simple functions are dense in $E(X)$, then $E(X)$ is also LOH/LASQ.
\item If $X$ is LOH/LASQ, then $L^{\infty}(\mu,X)$ is also LOH/LASQ.
\end{enumerate}
\par
The proofs in \cite{hardtke} are based on a general reduction theorem	and cor\-responding results for absolute sums of LOH/LASQ spaces that were obtained in \cite{abrahamsen2}*{Proposition 5.3}.\par
The assumption that the simple functions are dense in $E(X)$ holds true whenever $E$ is order continuous, in particular for $E=L^p(\mu)$ with $1\leq p<\infty$.\par
Here we will show that the result is still true without any additional assumptions on $E$ or $X$. The proof makes use of the
Kuratowski-Ryll-Nardzewski Theorem on the existence of measurable selections, which we will recall in the next section.

\section{Measurable selections}\label{sec:selections}
Let $(S,\A)$ be a measurable space and $(Y,d)$ a metric space.
Denote by $\Po(Y)$ the power-set of $Y$ and by $\B(Y)$ the Borel-$\sigma$-Algebra of $(Y,d)$. For a subset $M\ssq Y$ we denote by $\overline{M}$ the closure of $M$ in $(Y,d)$.\par
Let $F:S \rightarrow \Po(Y)$ be a set-valued map. For $M\ssq Y$ we put
\begin{equation*}
F_{-}(M):=\{s\in S:F(s)\cap M\neq \emptyset\}.
\end{equation*}
$F$ is said to be weakly $\A$-measurable if $F_{-}(U)\in \A$ for every open set $U\ssq Y$. If $F_{-}(C)\in \A$ for every closed set $C\ssq Y$, then $F$ is called $\A$-measurable. $\A$-measurability of $F$ implies weak $\A$-measurability (this follows from the fact that every open subset in a metric space is an $F_{\sigma}$-set).\par
The following standard lemma follows directly from the definition.
\begin{lemma}\label{lemma:closure}
Suppose that $F,G:S \rightarrow \Po(Y)$ are two set-valued maps such that $F(s)\ssq G(s)\ssq \overline{F(s)}$ for every $s\in S$. Then $F$ is weakly $\A$-measurable if ond only if $G$ is weakly $\A$-measurable.
\end{lemma}

The next lemma is also standard but we include a sketch of the proof here for the reader's convenience.
\begin{lemma}\label{lemma:caratheodory}
Suppose that $(Y,d)$ is separable and $g:S\times Y \rightarrow \R$ is a Carath\'eodory function, i.\,e. 
\begin{enumerate}[(i)]
\item $g(s,\cdot)$ is continuous for every $s\in S$,
\item $g(\cdot,y)$ is $\A$-measurable for every $y\in Y$.
\end{enumerate} 
Let $\alpha\in \R$ and put
\begin{equation*}
F(s):=\set*{y\in Y:g(s,y)> \alpha} \ \ \ \ \text{for\ all}\ s\in S.
\end{equation*}
Then $F:S \rightarrow \Po(Y)$ is $\A$-measurable.
\end{lemma}

\begin{proof}
Let $C\ssq Y$ be nonempty and closed. Choose a sequence $(y_n)_{n\in \N}$ such that $C=\overline{\set*{y_n:n\in \N}}$ and put $A_n:=\set*{s\in S:g(s,y_n)>\alpha}$ for every $n\in \N$. Then we have $A_n\in \A$	for each $n$ and it is easy to see that $F_{-}(C)=\bigcup_{n\in \N}A_n$. Thus $F_{-}(C)\in \A$.
\end{proof}	

A classical result on the existence of measurable selections is the Kuratowski-Ryll-Nardzewski Theorem (see for instance \cite{graf}*{Theorem 2.1}).
\begin{theorem}[Kuratowski-Ryll-Nardzewski Selection Theorem]\label{thm:kuratowski} 
Let $(S,\A)$ be a measurable space and $(Y,d)$ a complete, separable metric space. Let $F:S \rightarrow \Po(Y)$
be a weakly $\A$-measurable set-valued map such that $F(s)$ is non-empty and closed in $Y$ for every $s\in S$. Then there is an $\A$-$\B(Y)$-measurable map $f:S \rightarrow Y$ such that $f(s)\in F(s)$ for every $s\in S$.
\end{theorem}

\section{LOH and LASQ K\"othe-Bochner spaces}\label{sec:LOHLASQ}
Now we are ready to prove the general stability result.
\begin{theorem}\label{thm:lohlasq}
Let $(S,\A,\mu)$ be a complete, $\sigma$-finite measure space and let $E$ be a K\"othe function space over $(S,\A,\mu)$. If $X$ is a real Banach space which is LOH/LASQ, then the K\"othe-Bochner space $E(X)$ is also LOH/LASQ.
\end{theorem}

\begin{proof}
1) Assume that $X$ is LOH. Fix $\eps>0$ and $f\in E(X)$ with $\norm{f}_{E(X)}=1$.
Since $f:S \rightarrow X$ is Bochner-measurable there is a separable, closed subspace $Y$ of $X$ such that $f(s)\in Y$ for $\mu$-a.\,e. $s\in S$ (Pettis Measurability Theorem). Without loss of generality we may assume that this holds even for all $s\in S$.\par
Put $S^{\prime}:=\{s\in S:f(s)\neq 0\}$.\par
Since $Y$ is separable the unit sphere $S_Y$ is also separable. We fix a sequence $(y_n)_{n\in \N}$ which is dense in $S_Y$.\par
Because $X$ is LOH we can find a sequence $(z_n)_{n\in \N}$ in $S_X$ such that $\norm{y_n\pm z_n}\geq 2-\eps/2$ for every $n\in \N$. We put $Z:=\overline{\mathrm{span}}\{z_n:n\in \N\}$. This is again a separable, closed subspace of $X$.\par
Next we define $F:S^{\prime}\rightarrow \Po(S_Z)$ by
\begin{equation*}
F(s):=\overline{\set*{z\in S_Z:\norm*{\frac{f(s)}{\norm{f(s)}}\pm z}> 2-\eps}} \ \ \ \text{for\ all}\ s\in S^{\prime}.
\end{equation*}	
We have $F(s)\neq \emptyset$ for every $s\in S^{\prime}$. To see this note that $f(s)/\norm{f(s)}\in S_Y$ and hence there is some index $N\in \N$ such that $\norm{f(s)/\norm{f(s)}-y_N}<\eps/2$. It follows that
\begin{equation*}
\norm*{\frac{f(s)}{\norm{f(s)}}\pm z_N}\geq	\norm{y_N\pm z_N}-\norm*{\frac{f(s)}{\norm{f(s)}}-y_N}> 2-\frac{\eps}{2}-\frac{\eps}{2}=2-\eps.
\end{equation*}	  
If we define 
\begin{equation*}
g(s,z):=\min\set*{\norm*{\frac{f(s)}{\norm{f(s)}}+z},\norm*{\frac{f(s)}{\norm{f(s)}}-z}} \ \ \ \text{for}\ s\in S^{\prime},\,z\in S_Z,
\end{equation*}
then $F(s)=\overline{\set*{z\in S_Z:g(s,z)>2-\eps}}$ and since $g$ is a Carath\'eodory function, it follows from Lemma \ref{lemma:caratheodory} and Lemma \ref{lemma:closure} that $F$ is weakly measurable.\par
Thus by the Kuratowski-Ryll-Nardzewski Selection Theorem there exists a measurable function $\tilde{f}:S^{\prime}\rightarrow S_Z$ such that $\tilde{f}(s)\in F(s)$ for every $s\in S^{\prime}$. Note that by Pettis Measurability Theorem $\tilde{f}$ is also Bochner-measurable. We put $h(s):=\norm{f(s)}\tilde{f}(s)$ for $s\in S^{\prime}$ and $h(s):=0$ for $s\in S\sm S^{\prime}$. Then $h$ is Bochner-measurable and $\norm{h(s)}=\norm{f(s)}$ for every $s\in S$ and thus $\norm{h}_{E(X)}=\norm{f}_{E(X)}=1$.\par 
We further have $\norm{f(s)\pm h(s)}\geq (2-\eps)\norm{f(s)}$ for every $s\in S$. This implies $\norm{f\pm h}_{E(X)}\geq (2-\eps)\norm{f}_{E(X)}=2-\eps$.\par 
This proves that $E(X)$ is LOH.\par
2) Now assume that $X$ is LASQ. We take $\eps$, $f$, $Y$, $S^{\prime}$ and $(y_n)_{n\in \N}$ as in part 1). Since $X$ is LASQ there is a sequence $(z_n)_{n\in \N}$ in $S_X$ such that $\norm{y_n\pm z_n}\leq 1+\eps/2$ for every $n\in \N$ and we put $Z:=\overline{\mathrm{span}}\{z_n:n\in \N\}$ and 
\begin{equation*}
F(s):=\overline{\set*{z\in S_Z:\norm*{\frac{f(s)}{\norm{f(s)}}\pm z}<1+\eps}} \ \ \ \text{for\ all}\ s\in S^{\prime}.
\end{equation*} 
Analogously to the proof in part 1) we can see that each set $F(s)$ is non-empty and also that $F$ is weakly measurable. For the latter, put
\begin{equation*}
g(s,z):=\max\set*{\norm*{\frac{f(s)}{\norm{f(s)}}+z},\norm*{\frac{f(s)}{\norm{f(s)}}-z}} \ \ \ \text{for}\ s\in S^{\prime},\,z\in S_Z.	
\end{equation*}	
Then $g$ is a Carath\'eodory function and 
$F(s)=\overline{\set*{z\in S_Z:g(s,z)<1+\eps}}$.\par
By the Kuratowski-Ryll-Nardzewski Selection Theorem we find a measurable function
$\tilde{f}:S^{\prime} \rightarrow S_Z$ such that $\tilde{f}(s)\in F(s)$ for every $s\in S^{\prime}$. Then if we define $h$ as in 1), we find that $h\in S_{E(X)}$ and $\norm{f(s)\pm h(s)}\leq \norm{f(s)}(1+\eps)$ for every $s\in S$. This implies
$\norm{f\pm h}_{E(X)}\leq 1+\eps$ and the proof is finished.
\end{proof}

\address
\email
	
\end{document}